\documentclass[12pt]{article}

\title{Applying hypersurface bounds to a conjecture by Carlet}
\author{Zo\"{e} Gemmell and Tim Trudgian\footnote{Supported by Australian Research Council Discovery Project DP240100186.} \\
School of Science, UNSW Canberra, Australia \\
z.gemmell@unsw.edu.au \quad\quad timothy.trudgian@unsw.edu.au }

\usepackage{indentfirst}
\usepackage{url}
\usepackage{amsmath,amssymb,mathrsfs,amsthm}
\usepackage{comment}
\usepackage{fullpage}
\usepackage{hyperref}
\usepackage{booktabs}
\usepackage[dvipsnames]{xcolor} 

\newcommand{\FF}{\mathbb F}
\newcommand{\myrfloor}{\lfloor r \rfloor}
\newtheorem{theorem}{Theorem}
\newtheorem{lemma}[theorem]{Lemma}

\newtheorem{proposition}[theorem]{Proposition}

\begin{document}
\maketitle
\begin{abstract}
\noindent
A function from $\FF_{2^n}$ to $\FF_{2^n}$ is $k$th order sum-free if the sum of its values over each $k$-dimensional $\FF_2$-affine subspace is nonzero. It is conjectured that for $n$ odd and prime, $f_\textrm{inv}=x^{-1}$ is not $k$th order sum-free for $3 \leq k \leq n-3$. This is the unresolved part of Carlet's conjecture, which gives exact values for which $f_\textrm{inv}$ is $k$th order sum-free. We give two results as improvements on an explicit estimate on the number of $q$-rational points of an $\FF_q$-definable hypersurface previously proved by Cafure and Matera. We use these results to prove that $f_\textrm{inv}$ is not $k$th order sum-free for $3\leq k \leq \frac{3}{13}n+0.461$, improving on work previously done by Hou and Zhao.\\

\end{abstract}
\section{Introduction}
\noindent
Consider a function $f: \FF_{2^{n}} \rightarrow \FF_{2^{n}}$. We say that $f$ is `$k$th order sum-free' if, for every  $k$-dimensional affine subspace $X$ of $F_{2^{n}}$, we have $\sum_{x\in X}f(x) \neq 0$. Carlet first introduced this concept as a generalisation of almost perfect nonlinear (APN) functions in \cite{Carlet2}, as APN's are 2nd order sum-free functions. These functions are used in cryptography owing to their low differential uniformity and nonlinearity (see \cite{Blondeau}, Chapter 11 of \cite{Carlet1}, and the introduction of \cite{Charpin}).

Of particular interest is the function $f_{\textrm{inv}}(x)$, defined as 
\begin{equation*}
    f_{\textrm{inv}}(x) = \begin{cases}
        x^{-1}, & x \in \FF^*_{2^n} \\
        0, & x=0. \\
    \end{cases}
\end{equation*}

Carlet's conjecture is the statement that $f_{\textrm{inv}}$ is $k$th order sum-free, denoted $\textrm{SF}_n$, only for the following values of $k$:
\begin{equation*}
    \textrm{SF}_n = \begin{cases}
        k \in \{1,n-1\}, & \textrm{for even } n\\
        k \in \{1,2,n-2,n-1\}, & \textrm{for odd } n\geq3. \\
    \end{cases}
\end{equation*}

This can be rephrased as $f_{\textrm{inv}}$ is not $k$th order sum-free for $3\leq k \leq n-3$ when $n$ is odd and for $2\leq k \leq n-2$ when $n$ is even. The case of even $n$ was proved in \cite[Theorem 5.2]{Carlet4}; the case where $n$ is odd and composite was proved in in \cite{Hou2}. Many other partial results are known, we refer the reader to \cite{Hou1} for a summary.

We wish to focus on one such result in this short note. Hou and Zhao \cite[Theorem 5.1]{Hou1} prove that $f_{\textrm{inv}}$ is not $k$th order sum-free for $3\leq k \leq 3n/13 + 0.3161.$ One of the key ingredients in their proof is a seminal result by Cafure and Matera \cite{CM}. We improve this latter result slightly, which may be of independent interest. For example, our work could be used to improve the estimate for an arbitrary $\FF_q$-hypersurface (Theorem 5.7 of \cite{CM}). We will explore these applications in future work. 

We state our point-counting result in Theorem \ref{GemmellTrudgian}, which allows us to prove our main contribution below.
\begin{theorem} \label{blue}
    The function $f_{\textrm{inv}}$ is not $k$th order sum-free for $3\leq k \leq 3n/13+0.461$.
\end{theorem}
We need to introduce some background on the result from Cafure and Matera bounding the number of $q$-rational points on an absolutely irreducible $\FF_q$-definable hypersurface $H$. See (Section 1 of \cite{CM}) for a thorough background. Let $p$ be a prime number and $q:=p^k$ where $k>0$. Let $\FF_q$ be the finite field of $q$ elements and let $\overline{\mathbb{F}}_q$ be the algebraic closure of $\FF_q$. Finally let $F_1,\ldots,F_m \in \FF_q[X_1,\ldots,X_n]$ be a set of polynomials and let $V$ be the affine subvariety of $\overline{\mathbb{F}}_q^n$ defined by the polynomials $F_1,\ldots,F_m$.

We seek to improve the following result for an absolutely irreducible $\FF_q$-hypersurface $H \subset \overline{\mathbb{F}}_q^n$ of degree $\delta$ \cite[Theorem 5.2]{CM}:
\begin{equation*}
    |\#(H\cap\FF_q^n)-q^{n-1}| \leq (\delta-1)(\delta-2)q^{n-\frac{3}{2}}+5\delta^{\frac{13}{3}}q^{n-2}.
\end{equation*}

To do this we improve an earlier result of Cafure and Matera on the  number of restrictions of a given absolutely irreducible polynomial $f\in\FF_q[X_1,\ldots,X_n]$ to affine planes $L$ having a fixed number of absolutely irreducible factors over $\FF_q$, presented in Section 2. In Section 3 we prove the improved results for an absolutely irreducible $\FF_q$-hypersurface $H \subset \overline{\mathbb{F}}_q^n$ of degree $\delta$. Finally in Section 4 we apply these results to prove Theorem 1.

\section{The number of planes for which $f_L$ has a fixed number of absolutely irreducible $\FF_q$-factors}

Let $f$ be a polynomial where $f\in \FF_q[X_1,\ldots,X_n]$ of degree $\delta>0$. Let $L$ be a plane and $f_L$ be the restriction of $f$ to $L$. Cafure and Matera proved the following result.
\begin{proposition}\label{CandMProp4.1}
    (\cite[Proposition 4.1]{CM})Let $f \in \FF_{q}[X_1,\ldots,X_n]$ be a polynomial of degree $\delta >1$. We have
    \begin{equation*}
        \sum_{j=1}^{\delta-1}j\#(\Pi_j)\leq (2\delta^{\frac{13}{3}}+3\delta^{\frac{11}{3}})\frac{q^{3n-3}}{q^3(q-1)}.
    \end{equation*}
\end{proposition}

We improved Cafure and Matera's result below with the exception of some values of $\delta$. We also give an improvement that applies to all values of $\delta$ in Lemma \ref{GemmellTrudgianLemma}.
\begin{theorem}\label{GemmellTrudgian}
    Let $f \in \FF_{q}[X_1,\ldots,X_n]$ be a polynomial of degree $\delta >1$ and $\delta \notin [6,37]$. We have
    \begin{equation*}
        \sum_{j=1}^{\delta-1}j\#(\Pi_j)\leq (1.99\delta^{\frac{13}{3}})\frac{q^{3n-3}}{q^3(q-1)}.
    \end{equation*}
    Where $\Pi_j$ is the set of planes L for which $f_L$ has j+1 absolutely irreducible $\FF_q$-factors.
\end{theorem}
\begin{proof}
    The first part of this proof is taken from \cite{CM}. For $\delta=2$ the expression $\sum_{j=1}^{\delta-1}j\#(\Pi_j)$ consists of only one term, namely $\Pi_1$ and therefore the result in Corollary 3.2 in \cite{CM} yields
    \begin{equation*}
        \#(\Pi_1)\leq\left(\frac{3}{2}\delta^4-2\delta^3+\frac{5}{2}\delta^2\right)\frac{q^{3n-6}}{(q-1)} \leq (0.893\delta^{13/3})\frac{q^{3n-6}}{(q-1)}.
    \end{equation*}

    From now on we may assume $\delta\geq3$. Let $r$ be a real number in the open interval $(1,\delta-1)$. From the expansion of $\sum_{j=1}^{\delta-1}j\#(\Pi_j)$, Corollary 3.2 and $(11)$ in \cite{CM}, we have
    \begin{equation}\label{purple}
        \sum_{j=1}^{\delta-1}j\#(\Pi_j) \leq \underbrace{\left(r\left(\frac{3}{2}\delta^4-2\delta^3+\frac{5}{2}\delta^2\right) + \left( \delta^5c_1+3\delta^4c_2+\delta^3c_3-\frac{3}{4}\delta^2c_4+2\delta^2\right)\right)}_{F}\frac{q^{3n-6}}{(q-1)}
    \end{equation}
    where $c_1,c_2,c_3,c_4$ are:
    \begin{align}\label{constants}
        c_1 &:= \sum_{j=\lfloor r\rfloor+1}^{\delta-1}\frac{1}{j^3}-\frac{1}{8j^4}, \quad c_2 := \sum_{j=\lfloor r\rfloor+1}^{\delta-1}\frac{1}{j^2}-\frac{1}{4j^3}, \quad c_3 := \sum_{j=\lfloor r\rfloor+1}^{\delta-1}\frac{2}{j}-\frac{11}{8j^2}, \quad c_4 := \sum_{j=\lfloor r\rfloor+1}^{\delta-1}\frac{1}{j}.
    \end{align}
    
    For $\delta \in [3, 3\times10^5]$ and $\delta \notin [6,37]$ we verify that
$        F \leq 1.99\delta^{\frac{13}{3}}$,
by computation. For example
    \begin{center}
      \begin{tabular}{|c|c|c|c|}
          \hline
          $\delta$ & $\textrm{Optimal } r$ & \textrm{Computed Bound} & $1.99\delta^{\frac{13}{3}}$\\
          \hline
          $5$  & $2$ & $2114.112\ldots$ & $2126.782\ldots$ \\
          $6$  & $2$ & $4755.719\ldots$ & $4686.426\ldots$ \\
          $37$ & $3$ & $12430121.798\ldots$ & $12427789.273\ldots$ \\
          $38$ & $3$ & $13920945.841\ldots$ & $13950250.14$  \\
          \hline
      \end{tabular}
\end{center}
    
    Upper bounds for $\delta \in [6,37]$ can be found in Appendix A.

    For $\delta \geq 3\times10^5$ we use a different approach. For any decreasing positive real function $f$ we can bound the sum of the function by integrals
    \begin{equation*}
        \int_a^{b+1}f(x)dx \leq \sum_{j=a}^bf(j) \leq \int_{a-1}^bf(x)dx.
    \end{equation*}
    Cafure and Matera use the $r$ as the lower limit of the integral, we use the sharper limit of $\myrfloor$ to improve on the upper bounds for $c_1, c_2, c_3$ and obtain a lower bound for $c_4$
    \begin{align*}
       c_1 &\leq \int_{\myrfloor}^{\delta-1}\frac{1}{x^3}-\frac{1}{8x^4}dx, \quad c_2 \leq \int_{\myrfloor}^{\delta-1}\frac{1}{x^2}-\frac{1}{4x^3}dx,\\
       c_3 &\leq \int_{\myrfloor}^{\delta-1}\frac{2}{x}-\frac{11}{8x^2}dx, \quad 
       c_4 \geq \int_{\myrfloor + 1}^{\delta}\frac{1}{x}dx 
    \end{align*}
    
    Let $\beta := \frac{\delta}{(\delta-1)}$, then for $\delta \geq 3\times10^5$ we have $1 < \beta < 1+4\times10^{-6}$, and 
    \begin{align}\label{brown}
        \delta^5c_1 &\leq -\frac{\delta^3}{2}+\frac{\beta^3\delta^2}{24}+\frac{\delta^5}{2(\myrfloor)^2}-\frac{\delta^5}{24(\myrfloor)^3} \\
        \nonumber \\
        3\delta^4c_2 &\leq -3\delta^3+\frac{3\beta^3\delta^2}{8}+\frac{3\delta^4}{\myrfloor}+\frac{3\delta^4}{8(\myrfloor)^2} \\
        \nonumber \\
        \delta^3c_3 &\leq 2\log{\left(\frac{\delta-1}{\myrfloor}\right)}+\frac{11\beta\delta^2}{8}-\frac{11\delta^3}{8\myrfloor} \\
        \nonumber \\
        -\frac{3}{4}\delta^2c_4 &\leq -\frac{3}{4}\delta^2\log{\left(\frac{\delta}{\myrfloor+1}\right)}.
    \end{align}
    
    Let $F$ be as in (\ref{purple}), then we obtain the following
    \begin{align*}
        F = \delta^5c_1+\frac{3}{2}(\lambda\delta^{\frac{1}{3}})\delta^4+3\delta^4c_2-2(\lambda\delta^{\frac{1}{3}})\delta^3+\delta^3c_3+\frac{5}{2}(\lambda\delta^{\frac{1}{3}})\delta^2-\frac{3}{4}\delta^2c_4+2\delta^2 \\
        \leq \left(\frac{1}{2(\myrfloor)^2}-\frac{1}{24(\myrfloor)^3}\right)\delta^5 + \left(\frac{3\lambda}{2}\right)\delta^{\frac{13}{3}} + \left(\frac{3}{\myrfloor}-\frac{3}{8(\myrfloor)^2}\right)\delta^4 - 2\lambda\delta^{\frac{10}{3}} \\
        + \left(2\log{\left(\frac{\delta-1}{\myrfloor}\right)}-\frac{7}{2}-\frac{11}{8\myrfloor}\right)\delta^3 + \left(\frac{5\lambda}{2}\right)\delta^{\frac{7}{3}} \\
        + \left(\frac{\beta^3}{24}+\frac{3\beta^2}{8}+\frac{11\beta}{8}+2-\frac{3}{4}\log{\left(\frac{\delta}{\myrfloor+1}\right)}\right)\delta^2.
    \end{align*}

    Since $-\frac{3}{4}\log{\left(\frac{\delta}{\myrfloor+1}\right)}$ is a decreasing function and which is negative at $\delta=3\times10^5$, so this can be removed. Using $r-1 < \myrfloor\leq r < \myrfloor+1$, we get
    
    \begin{align*}
        F \leq \left(\frac{1}{2(r-1)^2}-\frac{1}{24(r-1)^3}\right)\delta^5 + \left(\frac{3\lambda}{2}\right)\delta^{\frac{13}{3}} + \left(\frac{3}{r-1}-\frac{3}{8(r-1)^2}\right)\delta^4 - 2\lambda\delta^{\frac{10}{3}} \quad\quad\quad \\
        + \left(2\log{\left(\frac{\delta-1}{r-1}\right)}-\frac{7}{2}-\frac{11}{8(r-1)}\right)\delta^3 + \left(\frac{5\lambda}{2}\right)\delta^{\frac{7}{3}}
        + \left(\frac{\beta^3}{24}+\frac{3\beta^2}{8}+\frac{11\beta}{8}+2\right)\delta^2.
    \end{align*}

    Let $r:=\lambda\delta^{\frac{1}{3}}$, to obtain the optimum value of $\lambda$ we must minimise the leading coefficient of the upper bound of $F$:
    \begin{align*}
         & \left(\frac{1}{2(r-1)^2}-\frac{1}{24(r-1)^3}\right)\delta^5 + \left(\frac{3\lambda}{2}\right)\delta^{\frac{13}{3}} + \ldots \\
         &= \frac{\delta^5}{2(r-1)^2}-\frac{\delta^5}{24(r-1)^3}+\frac{3\lambda}{2}\delta^{\frac{13}{3}} + \ldots \\
         &= \frac{\delta^5}{2r^2(1-\frac{1}{r})^2}+\frac{3\lambda}{2}\delta^{\frac{13}{3}}-\frac{\delta^5}{24r^3(1-\frac{1}{r})^3} + \ldots \\
         &= \left(\frac{1}{2\lambda^2}\times\frac{1}{(1-\frac{1}{r})^2}+\frac{3\lambda}{2}\right)\delta^{\frac{13}{3}} - \ldots
    \end{align*}
    Since $r>0$ we have $1-\frac{1}{r}<1$, we want to minimise:
    \begin{equation*}
        \left(\frac{1}{2\lambda^2}+\frac{3\lambda}{2}\right)\delta^{\frac{13}{3}}
    \end{equation*}
    Giving $\lambda=\sqrt[3]{\frac{2}{3}}$.

    Let $\gamma := \frac{r}{r-1}$, then we have $1 < \gamma < 1.02$ and we can further simplify the bound:
    \begin{align}\label{gold}
        F \leq \left(\frac{\gamma^2}{2\lambda^2} + \frac{3\lambda}{2}\right)\delta^{\frac{13}{3}}-\frac{\gamma^3}{24\lambda^3}\delta^4+\frac{3\gamma}{\lambda}\delta^{\frac{11}{3}}+\left(\frac{5}{2}-\frac{3\gamma^2}{8\lambda^2}-2\lambda\right)\delta^{\frac{10}{3}} \quad\quad\quad\quad\quad\quad\quad \nonumber \\
        + 2\left(\log{\left(\frac{\gamma(\delta-1)}{\lambda\delta^{\frac{1}{3}}}\right)}-\frac{7}{4}\right)\delta^3 -\frac{11\gamma}{8\lambda}\delta^{\frac{8}{3}}+\frac{5\lambda}{2}\delta^{\frac{7}{3}}+\left(\frac{\beta^3}{24}+\frac{3\beta^2}{8}+\frac{11\beta}{8}+2\right)\delta^2
    \end{align}

    When $\delta$ is large, $\delta \geq 3\times10^{5}$ is enough, we can bound $F$ in the following way:
    \begin{equation*}
        F \leq \left(\frac{\gamma^2}{2\lambda^2} + \frac{3\lambda}{2}\right)\delta^{\frac{13}{3}}\approx (1.98855\ldots)\delta^{\frac{13}{3}}.
    \end{equation*}
    
    Therefore we obtain our result
    \begin{equation*}
        \sum_{j=1}^{\delta-1}j\#(\Pi_j)\leq (1.99\delta^{\frac{13}{3}})\frac{q^{3n-3}}{q^3(q-1)},
    \end{equation*}
    for $\delta\geq2$ where $\delta\notin[6,37]$.
\end{proof}

Computationally we checked every integer option of $r \in (1,\delta-1)$ for $\delta\in [3,3\times10^5]$, taking advantage of vectorization of addition in Python. With more computing power the upper bound on the $\delta$'s computationally checked could be increased, call this $M$. This would bring down the upper bound on $\gamma$, in turn lowering the value of the coefficient of our leading term $\left(\frac{\gamma^2}{2\lambda^2} + \frac{3\lambda}{2}\right)$. Irrespective of large $M$ since $\gamma > 1$ we cannot do any better than $(1.96555\ldots)\delta^{\frac{13}{3}}$.

With our computing power, checking $\delta \in [3,10^5]$ took around 20 minutes, $\delta \in [3, 2\times10^5]$ took just under an hour and $\delta \in [3,3\times10^5]$ took two and a half hours to confirm the result in Theorem \ref{GemmellTrudgian}. To obtain the best possible upper bound to two decimal places of $1.97\delta^{\frac{13}{3}}$ we would need to check the values of $\delta \in [3,4\times10^7]$.

After checking computationally we get $2.043\delta^{\frac{13}{3}}$ as an upper bound for every $\delta\geq2$. This is the best that we can reasonably do to 3 decimal places, leading to the following result.

\begin{lemma} \label{GemmellTrudgianLemma}
    Let $f \in \FF_{q}[X_1,\ldots,X_n]$ be a polynomial of degree $\delta >1$. We have
    \begin{equation*}
        \sum_{j=1}^{\delta-1}j\#(\Pi_j)\leq (2.043\delta^{\frac{13}{3}})\frac{q^{3n-3}}{q^3(q-1)}.
    \end{equation*}
    Where $\Pi_j$ is the set of planes L for which $f_L$ has j+1 absolutely irreducible $\FF_q$-factors.
\end{lemma}
\begin{proof}
    For $\delta \in [6,37]$ this was checked computationally, for $\delta \in [2,5]$ and $\delta \geq 38$ this follows from the proof of Theorem \ref{GemmellTrudgian}.
\end{proof}
This result is improves on Cafure and Matera's result (stated in Proposition \ref{CandMProp4.1}) when applied in cases such as Carlet's Conjecture, where the following upper bound is used
\begin{equation*}
    \sum_{j=1}^{\delta-1}j\#(\Pi_j)\leq (2\delta^{\frac{13}{3}}+3\delta^{\frac{11}{3}})\frac{q^{3n-3}}{q^3(q-1)} \leq (5\delta^{\frac{13}{3}})\frac{q^{3n-3}}{q^3(q-1)}.
\end{equation*}

\section{Estimate on the number of $q$-rational points of a given absolutely irreducible $\FF_q$-hypersurface}
In their paper Cafure and Matera use the constant of their bound $(2\delta^{\frac{13}{3}}+3\delta^{\frac{11}{3}})$ to obtain estimates on the number of $q$-rational points of a given absolutely irreducible $\FF_q$-hypersurface, we insert our bound from Theorem \ref{GemmellTrudgian} in place of this to obtain the following result.

\begin{theorem}\label{HypersurfaceEstimate}
    For an absolutely irreducible $\FF_q$-hypersurface H of $\mathbb{A}^n$ of degree $\delta \notin [6,37]$ the following estimate holds:
    \begin{equation*}
        |\#(H\cap\FF_q^n)-q^{n-1}| \leq (\delta-1)(\delta-2)q^{n-\frac{3}{2}}+2.86\delta^{\frac{13}{3}}q^{n-2}.
    \end{equation*}
\end{theorem}

\begin{proof}
    This follows from Cafure and Matera's proof of Theorem 5.2 \cite{CM}, the only change that was made is detailed as follows:

     Cafure and Matera \cite[p.\ 169]{CM} introduce variables $A$-$E$. We have improved upon the upper bound for $B$ in Theorem \ref{GemmellTrudgian} and therefore make the following change to $\frac{B}{A}$:
    \begin{equation*}
        B:=\sum_{j=1}^{\delta-1}j(\#\Pi_j), \quad
        \frac{B}{A}\leq (1.99\delta^{\frac{13}{3}})\frac{q^{n-2}}{q^{n-1}-1}.
    \end{equation*}

    Using this updated bound in (22) of Cafure and Matera's proof gives the following result for $\delta\geq3$ and $\delta \notin [6,37]$:
    \begin{align} \label{yellow}
        |N-q^{n-1}| &\leq q^{n-2}\left((\delta-1)(\delta-2)q^\frac{1}{2} + \delta + 1 + \delta^2 + (1.99\delta^{\frac{13}{3}})\frac{q^{n-1}}{q^{n-1}-1}+\frac{\delta^2(q-1)}{q}+\frac{4}{3}\right) \nonumber \\ 
        &\leq q^{n-2}\left((\delta-1)(\delta-2)q^\frac{1}{2} + \delta + 1 + \delta^2 + \frac{4}{3}(1.99\delta^{\frac{13}{3}})+\delta^2+\frac{4}{3}\right) \nonumber \\
        &\leq (\delta-1)(\delta-2)q^{n-\frac{3}{2}} + 2.86\delta^{\frac{13}{3}}q^{n-2}.
    \end{align}

    Additionally the result holds for $\delta=2$:
    \begin{align*}
        |N-q^{n-1}| &\leq q^{n-2}\left((\delta-1)(\delta-2)q^\frac{1}{2} + \delta + 1 + \delta^2 + \frac{4}{3}\left(\frac{3}{2}\delta^4-2\delta^3+\frac{5}{2}\delta^2\right)+\delta^2+\frac{4}{3}\right) \nonumber \\
        &\leq (1.803\delta^{\frac{13}{3}})q^{n-2}.
    \end{align*}
\end{proof}

Again we can obtain a result that holds for all $\delta$.
\begin{lemma}\label{HypersurfaceEstimateLemma}
    For an absolutely irreducible $\FF_q$-hypersurface H of $\mathbb{A}^n$ of degree $\delta$ we have
        \begin{equation*}
        |\#(H\cap\FF_q^n)-q^{n-1}| \leq (\delta-1)(\delta-2)q^{n-\frac{3}{2}}+G\delta^{\frac{13}{3}}q^{n-2},
    \end{equation*}
    where $G = 2.924$ for $\delta\geq 2$, and $G = 2.741$ for $\delta \geq 8$.
\end{lemma}

\begin{proof}
    Case: $\delta \geq 8$ we take $G=2.741$. We take $\delta\geq8$ in (\ref{yellow}) and obtain the result.
    
    Case: $\delta \geq 2$ we take $G = 2.924$.
    
    For $\delta \in [6,37]$ the same logic was used as in the proof of Theorem \ref{HypersurfaceEstimate} except using the improvement to the bound of $B$ from Lemma \ref{GemmellTrudgianLemma} and make the following change to $\frac{B}{A}$:
    \begin{equation*}
        B:=\sum_{j=1}^{\delta-1}j(\#\Pi_j), \quad
        \frac{B}{A}\leq (2.043\delta^{\frac{13}{3}})\frac{q^{n-2}}{q^{n-1}-1}.
    \end{equation*}

    Producing the following result for $\delta\geq3$:
    \begin{align*}
        |N-q^{n-1}| &\leq q^{n-2}\left((\delta-1)(\delta-2)q^\frac{1}{2} + \delta + 1 + \delta^2 + (2.043\delta^{\frac{13}{3}})\frac{q^{n-1}}{q^{n-1}-1}+\frac{\delta^2(q-1)}{q}+\frac{4}{3}\right) \nonumber \\ 
        &\leq q^{n-2}\left((\delta-1)(\delta-2)q^\frac{1}{2} + \delta + 1 + \delta^2 + \frac{4}{3}(2.043\delta^{\frac{13}{3}})+\delta^2+\frac{4}{3}\right) \nonumber \\
        &\leq (\delta-1)(\delta-2)q^{n-\frac{3}{2}} + 2.924\delta^{\frac{13}{3}}q^{n-2}.
    \end{align*}

    For $\delta = 2$ this follows from the proof of Theorem \ref{HypersurfaceEstimate}.
\end{proof}

\section{Improved results on Carlet's conjecture}
We are now able to increase the bound on $n$ for which $f_{\textrm{inv}}$ is $k$th order sum-free.

\begin{theorem} \label{CarletImprovement}
    If $k\geq 3$ and $n \geq \frac{13}{3}k-2$
    then $f_{\textrm{inv}}$ is not kth order sum-free.
\end{theorem}

Before the proof of Theorem \ref{CarletImprovement} we need some additional notation. For a field $\FF$ and polynomial $f \in \FF[X_1,\ldots,X_k]$ define the zero set, or affine variety, $V_{\FF^k}$ as:
\begin{equation*}
    V_{\FF^k}(f) = \{(x_1,\ldots,x_k)\in\FF^k : f(x_1,\ldots,x_k)=0\}.
\end{equation*}
Additionally, let $\Lambda_k$ denote the set of all 2-adic partitions of $2^{k-1}$, where all of the parts of the partition are powers of 2, and define
\begin{equation*}
    \Theta_k(X_1,\ldots,X_k)=\sum_{\lambda\in\Lambda_k}m_{\lambda}(X_1,\ldots,X_k),
\end{equation*}
where $m_{\lambda}(X_1,\ldots,X_k)$ is the monomial symmetric polynomial associated to the partition $\lambda$ and the degree of $\Theta_k=2^{k-1}$. Finally, $\Delta(X_1,\ldots,X_k)$ is the Moore determinant over $\FF_2$. For details on the construction of $\Theta_k$ and $\Delta$ see \cite[Section 3]{Hou1}.

\begin{proof}
    This proof was inspired by the proofs of Theorem 4.2 of \cite{Carlet4} and Theorem 5.1 of \cite{Hou1}.
    
    Assume $k\geq4$ since $f_{\textrm{inv}}$ is not 3rd order sum free for $n\geq6$ \cite[Theorem 6]{Carlet3}. By Theorem 3.2 of \cite{Hou1} it is enough to show
    \begin{equation*}
        |V_{\FF_{2^n}^k}(\Theta_k) \backslash V_{\FF_{2^n}^k}(\Delta)|>0.
    \end{equation*} 
    Let $q=2^n$. Since $\Theta_k$ is absolutely irreducible, and $\delta = 2^{k-1}$ by Theorem \ref{HypersurfaceEstimate} and Lemma \ref{HypersurfaceEstimateLemma},
    \begin{equation*}
        |V_{\FF_{2^n}^k}(\Theta_k)| \geq q^{k-1}-(2^{k-1}-1)(2^{k-1}-2)q^{k-\frac{3}{2}}-G(2^{k-1})^{\frac{13}{3}}q^{k-2},
    \end{equation*}
    where $G$ is $2.924$ for $4\leq k \leq 7$ and $2.741$ when $k \geq 8$.\

    By Lemma 2.2 from \cite{CM}
    \begin{equation*}
        |V_{\FF_{2^n}^k}(\Theta_k) \cap V_{\FF_{2^n}^k}(\Delta)| \leq (2^{k-1})^2q^{k-2}<2^{2k}q^{k-2}.
    \end{equation*}

    Therefore
    \begin{align}\label{VarietyEquivalence}
        |V_{\FF_{2^n}^k}(\Theta_k) \backslash V_{\FF_{2^n}^k}(\Delta)| &\geq q^{k-1}-(2^{k-1}-1)(2^{k-1}-2)q^{k-\frac{3}{2}} - (G\times2^{\frac{13(k-1)}{3}} + 2^{2k})q^{k-2} \nonumber \\
        &= q^{k-2}(q-(2^{k-1}-1)(2^{k-1}-2)q^{\frac{1}{2}}
        - (G\times2^{\frac{13(k-1)}{3}} + 2^{2k})).
    \end{align}
    
    For $4 \leq k < 100$ we obtain explicit bounds for $n$ of the form $n\geq \frac{13}{3}k-J$ or $n> \frac{13}{3}k-J$ for which (\ref{VarietyEquivalence}) is positive. Since Carlet's conjecture has been proven true for odd and composite $n$ \cite{Hou2}, and for even $n$ \cite{Carlet4} and we only have to concern ourselves with $n$ which are odd and prime. For example when $k=4$,  $n\geq 20$ gives us (\ref{VarietyEquivalence}) $>0$, so we can take $n>19$ producing the bound $n>\frac{13}{3}k-4.3\ldots$. When $k=9$ we get our worst value of $J$: $n\geq \frac{13}{3}k-2$.

    For $k \geq 100$, let $y:=q^{\frac{1}{2}}=2^{\frac{n}{2}}$, then we can rewrite (\ref{VarietyEquivalence}) as
    \begin{equation}\label{black}
        |V_{\FF_{2^n}^k}(\Theta_k) \backslash V_{\FF_{2^n}^k}(\Delta)| > q^{k-2}(y^2-2^{2(k-1)}y-(2.86\times2^{\frac{13(k-1)}{3}} + 2^{2k})).
    \end{equation}
    
    Let $y_0$ represent the larger root of the quadratic (\ref{black}) we get:
    \begin{align*}
        y_0 &= \frac{1}{2}\left(2^{2(k-1)} + \sqrt{2^{4(k-1)}+4\times2.86\times2^{\frac{13(k-1)}{3}}+2^{2k+2}}\right) \\
        &= \frac{1}{2}\left(2^{-\frac{(k-1)}{6}} + \sqrt{2^{-\frac{(k-1)}{3}}+11.44+2^{\frac{19-7k}{3}}}\right)\times2^{\frac{13(k-1)}{6}} \\
        &\leq \left(\frac{1}{65536\times\sqrt{2}} + \sqrt{\frac{1}{2^{33}}+11.44+\frac{1}{2^{227}}}\right)\times2^{\frac{13(k-1)}{6}-1} \quad \quad (\textrm{since } k \geq 100).
    \end{align*}

    Hence it is sufficient to show that
    \begin{equation*}
        y=2^{\frac{n}{2}} \geq \left(\frac{1}{65536\times\sqrt{2}} + \sqrt{\frac{1}{2^{33}}+11.44+\frac{1}{2^{227}}}\right)\times2^{\frac{13(k-1)}{6}-1}
    \end{equation*}

    Rewriting this in terms of $n$ gives:
    \begin{align*}
        n &\geq \frac{13}{3}k + 2\log_2\left(\frac{1}{65536\times\sqrt{2}} + \sqrt{\frac{1}{2^{33}}+11.44+\frac{1}{2^{227}}}\right)-\frac{19}{3} \\
        &\geq\frac{13}{3}k -2.817.
    \end{align*}
\end{proof}
This condition can be written as $3\leq k \leq \frac{3}{13}n+0.461$, proving Theorem \ref{blue}.

\section*{Acknowledgements}
The first author thanks Daniel Petrov for his advice on improving the efficiency of repetitive operations in Python.

\newpage
\appendix
\section{Exceptions of Theorem \ref{GemmellTrudgian}}
The table below contains the optimal upper bounds for $\delta \in [6,37]$ calculated using:
\begin{equation}
    r\underbrace{\left(\frac{3}{2}\delta^4-2\delta^3+\frac{5}{2}\delta^2\right)}_{P} + \left( \delta^5c_1+3\delta^4c_2+\delta^3c_3-\frac{3}{4}\delta^2c_4+2\delta^2\right)
\end{equation}
where $c_1,\ldots,c_4$ are from (\ref{constants}). Each integer value of $r \in (1,\delta-1)$ was checked to choose $r$ optimally. Since $P$ is positive and increasing we can freely choose the value of $r$, therefore we took $r=\myrfloor$ and choose the value of $r$ which gives the smallest bound.

    \begin{center}
      \begin{tabular}{|c|c|c|c|}
          \hline
          $\delta$ & $\textrm{Choose } r$ & \textrm{Computed Bound} & $1.99\delta^{\frac{13}{3}}$\\
          \hline
          $6$   & $2$ & $4755.719\ldots$ &      $4686.426\ldots$     \\
          $7$   & $2$ & $9356.323\ldots$ &      $9139.966\ldots$     \\
          $8$   & $2$ & $16734.776\ldots$ &     $16302.08$           \\
          $9$   & $2$ & $27873.380\ldots$ &     $27158.385\ldots$    \\
          $10$  & $2$ & $43926.616\ldots$ &     $42873.250\ldots$    \\
          $11$  & $2$ & $66229.915\ldots$ &     $64796.972\ldots$    \\
          $12$  & $2$ & $96308.469\ldots$ &     $94472.442\ldots$    \\
          $13$  & $2$ & $135886.071\ldots$ &    $133641.375\ldots$   \\
          $14$  & $2$ & $186893.968\ldots$ &    $184250.170\ldots$   \\
          $15$  & $2$ & $251479.733\ldots$ &    $248455.452\ldots$   \\
          $16$  & $2$ & $332016.143\ldots$ &    $328629.339\ldots$   \\
          $17$  & $2$ & $431110.078\ldots$ &    $427364.459\ldots$   \\
          $18$  & $2$ & $551611.412\ldots$ &    $547478.747\ldots$   \\
          $19$  & $2$ & $696621.921\ldots$ &    $692020.054\ldots$   \\
          $20$  & $2$ & $869504.192\ldots$ &    $864270.569\ldots$   \\
          $21$  & $2$ & $1073890.535\ldots$ &   $1067751.082\ldots$  \\
          $22$  & $2$ & $1313691.905\ldots$ &   $1306225.105\ldots$  \\
          $23$  & $2$ & $1593106.818\ldots$ &   $1583702.853\ldots$  \\
          $24$  & $2$ & $1916630.275\ldots$ &   $1904445.097\ldots$  \\
          $25$  & $2$ & $2289062.689\ldots$ &   $2272966.913\ldots$  \\
          $26$  & $2$ & $2715518.813\ldots$ &   $2694041.310\ldots$  \\
          $27$  & $2$ & $3201436.665\ldots$ &   $3172702.77$         \\
          $28$  & $2$ & $3752586.464\ldots$ &   $3714250.685\ldots$  \\
          $29$  & $2$ & $4375079.559\ldots$ &   $4324252.720\ldots$  \\
          $30$  & $2$ & $5075377.363\ldots$ &   $5008548.076\ldots$  \\
          $31$  & $2$ & $5860300.290\ldots$ &   $5773250.692\ldots$  \\
          $32$  & $3$ & $6718968.141\ldots$ &   $6624752.368\ldots$  \\
          $33$  & $3$ & $7653256.765\ldots$ &   $7569725.817\ldots$  \\
          $34$  & $3$ & $8684471.693\ldots$ &   $8615127.653\ldots$  \\
          $35$  & $3$ & $9819513.439\ldots$ &   $9768201.324\ldots$  \\
          $36$  & $3$ & $11065570.052\ldots$ &  $11036479.977\ldots$ \\
          $37$  & $3$ & $12430121.798\ldots$ &  $12427789.273\ldots$ \\
          \hline
      \end{tabular}
\end{center}

\end{document}